\documentclass[11pt]{article}

\usepackage[T1]{fontenc}
\usepackage[utf8]{inputenc}
\usepackage{lmodern}
\usepackage{amsmath,amssymb,amsthm,mathtools}
\usepackage{microtype}
\usepackage[a4paper,margin=1in]{geometry}
\usepackage[colorlinks=true,linkcolor=blue,citecolor=blue,urlcolor=blue]{hyperref}

\newtheorem{theorem}{Theorem}[section]
\newtheorem{lemma}[theorem]{Lemma}
\newtheorem{proposition}[theorem]{Proposition}
\newtheorem{corollary}[theorem]{Corollary}
\newtheorem{remark}[theorem]{Remark}

\newcounter{maintheorem}
\renewcommand{\themaintheorem}{\arabic{maintheorem}}

\newenvironment{maintheorem}
{\refstepcounter{maintheorem}\par\medskip
 \noindent{\bf Theorem \themaintheorem.}\it}
{\par\medskip}

\newcommand{\F}{\mathbf F_2}
\newcommand{\wt}{\operatorname{wt}}
\newcommand{\supp}{\operatorname{supp}}
\newcommand{\Aut}{\operatorname{Aut}}
\newcommand{\Stab}{\operatorname{Stab}}
\newcommand{\St}{\operatorname{St}}
\newcommand{\PG}{\operatorname{PG}}
\newcommand{\PGL}{\operatorname{PGL}}
\newcommand{\punct}{\operatorname{punct}}
\newcommand{\1}{\mathbf 1}
\newcommand{\vareps}{\varepsilon}

\title{The unique self-dual binary code of length~26 \\ with minimum weight~6}

\author{Gerald H\"ohn\thanks{Department of Mathematics, Kansas State University. E-mail: \texttt{gerald@monstrous-moonshine.de}.}}

\date{}

\begin{document}
\maketitle

\begin{abstract}
The binary Type~I self-dual $[26,13,6]$ code is classical.  We give what appears to be the first direct non-computational proof of its uniqueness.  We first determine the weight enumerators of the code and its shadow.  Degree-one harmonic MacWilliams identities supply the required $1$-designs.  Elementary intersection counts then show that the two minimal half-shadows each contain $13$ words; these words label the $26$ coordinates as $13$ points and $13$ lines, and the two shadow classes become the point-stars and line-stars.  From this structure we give two uniqueness proofs: one reconstructs the projective plane of order~$3$ and the plane code, including the full automorphism group $\PGL(3,3)\rtimes C_2$; the other deletes an intrinsic flag, obtains the odd Golay code together with a deep hole coset, and reconstructs the length-$26$ code and the size of the automorphism group from this coset datum.  Thus, the natural length-$24$ object behind the code is the odd Golay code together with its unique orbit of deep hole cosets.
\end{abstract}



\section{Introduction}

A binary linear $[n,k,d]$ code is a $k$-dimensional subspace $C \leq \F^n$ whose nonzero words have minimum Hamming weight~$d$. 
It is \emph{self-dual} if $C=C^\perp$ with respect to the standard dot product
\[
        c\cdot d=\sum_{i=1}^n c_i d_i\in\F.
\]
Since $c\cdot c=\wt(c)\pmod 2$, every binary self-dual code is even.  A self-dual binary code is \emph{Type~II} if every codeword has weight divisible by $4$, and \emph{Type~I} otherwise.  For a Type~I code $C$ we write
\[
        C_0=\{c\in C:\wt(c)\equiv0\pmod4\}
\]
for the doubly-even subcode and
\[
        S=C_0^\perp\setminus C
\]
for the shadow.

\smallskip

The Type~I self-dual binary code of length $26$ and minimum weight $6$ has been known for a long time.  The surrounding enumeration program goes back to Pless--Sloane and Conway--Pless, and the revised enumeration of Conway, Pless, and Sloane records the uniqueness statement.  Pless obtained the code through the child construction from the doubly-even self-dual codes of length $32$; this is a highly structured enumeration proof rather than an intrinsic length-$26$ argument.    
Harada's weighing-matrix construction and Buyuklieva's circulant generator matrix give explicit construction data.  On the geometric side, Glynn's construction of self-dual codes from projective planes of odd order gives the same code for the plane of order~$3$; Dougherty later gave a related formulation. 
See~\cite{PlessSloane1975,ConwayPless1980,Pless1978,CPS1992,Harada1997,Buyuklieva1999,Glynn1991,Dougherty2005}.

\begin{maintheorem}\label{thm:main}
Up to equivalence there is a unique Type~I self-dual binary code of length $26$ and minimum weight $6$.  Its automorphism group is isomorphic to
\[
        \PGL(3,3)\rtimes C_2
\]
and has order $11232$.
\end{maintheorem}

The purpose of this note is to give a direct length-$26$ proof of this theorem.  To the author's knowledge, it is the first direct non-computational proof of the uniqueness of this code.
For binary self-dual codes, \emph{extremal} means attaining Rains' general bound \cite{Rains1998}, whereas \emph{optimal} means attaining the largest minimum distance possible among binary self-dual codes of the given length.  The present $[26,13,6]$ code is optimal but not extremal: the general bound at length~$26$ is~$8$.  Optimality alone does not force uniqueness; for instance, there are three and thirteen inequivalent optimal self-dual codes at lengths~$28$ and~$30$, respectively \cite{CPS1992}.  The uniqueness here is better understood through the recovered projective-plane geometry and, more specifically, the odd-Golay deep-hole construction.

\smallskip

The shadow formalism for singly-even self-dual codes was introduced by Conway and Sloane and further developed by Rains; for general background on self-dual codes and shadows, see also the survey of Rains and Sloane \cite{ConwaySloaneBound,Rains1998,RainsSloaneSelfDual}.  The design input comes from the Assmus--Mattson--Delsarte--Venkov circle of ideas.  Venkov used harmonic-weighted theta series to force spherical designs in extremal lattices, while Bachoc's harmonic weight enumerators provide the code-theoretic counterpart and a harmonic form of the Assmus--Mattson argument \cite{AssmusMattson1969,Delsarte1973,Venkov,Bachoc1999}.  The same extremal-design principle extends to conformal designs for vertex operator algebras \cite{HoehnConformal,HoehnShadow}; for the shadow setting see also Bachoc and Gaborit \cite{BachocGaborit2004}.

Starting with an arbitrary Type~I self-dual $[26,13,6]$ code $C$, Section~\ref{sec:enumerator} determines the weight enumerators of $C$ and its shadow, excludes the exceptional shadow case, and applies Bachoc's degree-one harmonic weight-enumerator identity to obtain the required $1$-designs.  Section~\ref{sec:projective-plane} uses elementary intersection counts to show that the two minimal half-shadows have $13$ words each, label the coordinates as $13$ points and $13$ lines, and identify the two shadow classes with point-stars and line-stars.  The resulting incidence structure is a projective plane of order~$3$, hence isomorphic to $\PG(2,3)$, and its intrinsic flags correspond bijectively to the weight-$6$ words.

Section~\ref{sec:plane-proof} recalls Glynn's plane code, proves that the recovered code is this code, and identifies the full automorphism group.  Section~\ref{sec:golay-proof} gives the structural reason for the existence of the code.  Deleting an intrinsic flag does not lead to the even Golay code, but to the odd Golay code $\mathcal G_{24}^{\mathrm{odd}}$, together with one of its deep hole cosets.  Conversely, adjoining two coordinates to $\mathcal G_{24}^{\mathrm{odd}}$ by means of a deep hole coset reconstructs a Type~I self-dual $[26,13,6]$ code.  Since the deep hole cosets form a single orbit under the sextet group, this construction gives a unique code.  In this sense the length-$26$ code is naturally the two-coordinate extension of the odd Golay code determined by its deep hole coset geometry.

\smallskip

This last viewpoint is also a code-theoretic analogue of the rank-$24$ deep hole description in Borcherds' thesis, where the unique extremal odd unimodular lattice in rank $26$ is found and studied~\cite{BorcherdsThesis}.
Subsequent work develops related reconstruction and marked-mass arguments for the rank-$26$ determinant-$3$ lattice and for definite integral Albert algebras \cite{HoehnHexagon26,HoehnAlbert26}.

It would also be natural to give a direct mass calculation for Type~I self-dual binary codes of length $26$ without weight-$2$ or weight-$4$ codewords, analogous to King's mass calculation for extremal doubly-even self-dual codes of length $40$~\cite{King01}.  Such a calculation should show that the projective-plane code already exhausts the mass.


\section{Weight enumerators and shadow designs}\label{sec:enumerator}

Let $C\le \F^{26}$ be Type~I self-dual with $d(C)=6$, let $C_0$ be its doubly-even subcode, and let $S=C_0^\perp\setminus C$ be its shadow.  Since $C_0$ has index $2$ in $C$ and index $4$ in $C_0^\perp$, we may write
\[
        C_0^\perp=C_0\sqcup C_1\sqcup C_2\sqcup C_3,
        \qquad C=C_0\sqcup C_2,
        \qquad S=C_1\sqcup C_3.
\]
For every Type~I self-dual code, the quotient $C_0^\perp/C_0$ is a four-group.  Thus, the sum of two vectors in the same shadow coset lies in $C_0$, whereas the sum of vectors in different shadow cosets lies in $C_2$.  We use this elementary coset arithmetic throughout.
For a binary word $u$, write $B_u=\supp(u)$, and write $e_i$ for the coordinate unit vector at $i$.
We write $W_D(x,y)$ for the ordinary weight enumerator of a code $D$.

\begin{proposition}\label{prop:parametric-enumerators}
There is a rational number $t$ such that
\[
\begin{aligned}
W_C(x,y)={}&x^{26}+y^{26}
 +(52-64\,t)\,(x^{20}y^6+x^6y^{20})
 +(390+320\,t)\,(x^{18}y^8+x^8y^{18})\\
& +(1313-576\,t)\,(x^{16}y^{10}+x^{10}y^{16})
 +(2340+320\,t)\,(x^{14}y^{12}+x^{12}y^{14}),
\end{aligned}
\]
and
\[
\begin{aligned}
W_S(x,y)={}&2\,t\,(x^{25}y+xy^{25})
 +(26-12\,t)\,(x^{21}y^5+x^5y^{21})\\
& +(1560+30\,t)\,(x^{17}y^9+x^9y^{17})
 +(5020-40\,t)\,x^{13}y^{13}.
\end{aligned}
\]
Moreover, $t\in\{0,\,\tfrac12\}$.
\end{proposition}

\begin{proof}
By Gleason's theorem for Type~I self-dual binary codes, with $A_i$ denoting the number of codewords of weight $i$, the conditions $A_0=1$ and $A_2=A_4=0$ give the displayed one-parameter weight enumerator.
The standard Conway--Sloane shadow transform gives the second display; see \cite[Ch.~19]{MS} and \cite{ConwaySloaneBound}.  Since $2\,t$ is the number of weight-$1$ shadow words and $A_6=52-64\,t\geq 0$, one has $t\in\{0,\,\tfrac12\}$.
\end{proof}

For a binary code $D\le\F^n$ and an arbitrary vector
\[
        a=(a_1,\,\ldots,\,a_n)\in\mathbb C^n,
        \qquad \sum_{i=1}^n a_i=0,
\]
put $a(X)=\sum_{i\in X}a_i$ and
\[
        W_{D,a}(x,y)=\sum_{d\in D}a(B_d)\,x^{n-\wt(d)}y^{\wt(d)}.
\]
A weight shell of $D$ is a $1$-design precisely when its coefficient in $W_{D,a}$ vanishes for every such $a$.  Bachoc's degree-one harmonic MacWilliams identity \cite{Bachoc1999} states that
\[
        W_{D,a}(x,y)=xy\,Z_{D,a}(x,y),
        \qquad
        Z_{D^\perp,a}(x,y)=-\frac{2}{|D|}Z_{D,a}(x+y,x-y).
\]
This is the harmonic-weight-enumerator form of the strength-one Assmus--Mattson argument \cite{AssmusMattson1969}.

\begin{proposition}\label{prop:four-weight-harmonic}
Let $D\le\F^n$, where $n=25$ or $26$.  If the nonzero weights of $D$ are contained in $\{8,\,12,\,16,\,20\}$ and $D^\perp$ has no nonzero word of weight less than $5$, then every weight shell of both $D$ and $D^\perp$ is a $1$-design.
\end{proposition}

\begin{proof}
Fix an arbitrary $a=(a_i)$ with $\sum_i a_i=0$.  The weight restriction gives
\[
        Z_{D,a}(x,y)=\sum_{w\in\{8,\,12,\,16,\,20\}}
        \alpha_w\, x^{n-w-1}y^{w-1}.
\]
Because $D^\perp$ has no words of weights $1$, $2$, $3$, or $4$, the coefficients of $x^{n-2-k}y^k$ in $Z_{D,a}(x+y,x-y)$ vanish for $k=0,1,2,3$.  With columns ordered by $w=8,12,16,20$, these conditions are $M_n(\alpha_8,\alpha_{12},\alpha_{16},\alpha_{20})^T=0$, where
\[
 M_{25}=\begin{pmatrix}
 1&1&1&1\\
 9&1&-7&-15\\
 29&-11&13&101\\
 21&-11&21&-395
 \end{pmatrix},\qquad
 M_{26}=\begin{pmatrix}
 1&1&1&1\\
 10&2&-6&-14\\
 38&-10&6&86\\
 50&-22&34&-294
 \end{pmatrix}.
\]
Indeed, the entry in row $k$ and column $w$ is
$\sum_{j=0}^k(-1)^j\binom{w-1}{j}\binom{n-w-1}{k-j}$.
Both matrices have determinant $2^{18}$.  Hence, $W_{D,a}=0$, and Bachoc's identity also gives $W_{D^\perp,a}=0$.  Since $a$ was arbitrary, all shells of both codes are $1$-designs.
\end{proof}

\begin{proposition}\label{prop:no-half}
The case $t=\tfrac12$ cannot occur.
\end{proposition}

\begin{proof}
Assume that $t=\tfrac12$ and, after permuting the coordinates, let $e_1$ be the unique weight-$1$ word in $S$.  Orthogonality to $C_0$ shows that the first coordinate is $0$ on $C_0$, and hence is constant on $C_2=C\setminus C_0$.  It must be $1$ there, since otherwise $e_1\in C^\perp=C$.  Thus, all $20$ weight-$6$ words of $C$ contain coordinate $1$.

Let $C'$ be obtained by puncturing $C$ at coordinate $1$.  Then $C'$ has minimum weight $5$, and its $20$ weight-$5$ words are precisely the punctures of the weight-$6$ words of $C$.  Moreover, $(C')^\perp$ is obtained by shortening $C^\perp=C$ at coordinate $1$.  Since the words of $C$ having first coordinate $0$ are exactly those in $C_0$, we have
$W_{(C')^\perp}(x,y)=x^{-1}W_{C_0}(x,y)$.
The enumerator of $C_0$ is obtained from the formula for $W_C$ in Proposition~\ref{prop:parametric-enumerators}, with $t=\tfrac12$, by retaining the terms whose weights are divisible by $4$.  Hence,
\[
\begin{aligned}
        W_{(C')^\perp}(x,y)
        ={}&x^{25}
          +550\,x^{17}y^8
          +2500\,x^{13}y^{12}
          +1025\,x^9y^{16}
          +20\,x^5y^{20},
\end{aligned}
\]
while its dual $C'$ has minimum weight $5$.  Proposition~\ref{prop:four-weight-harmonic}, applied to $D=(C')^\perp$,
therefore shows that the weight-$5$ words of $C'$ form a $1$-design on $25$ points.  Its replication number is
\[
        r=\frac{20\cdot 5}{25}=4.
\]

If $b\ne b'$ are two such words, their lifts lie in $C_2$ and contain coordinate $1$.  Their sum lies in $C_0$, so its weight is divisible by $4$ and at least $8$; hence $|B_b\cap B_{b'}|=1$.  The standard block-intersection identity for a fixed $b$ gives
\[
        19=\sum_{b'\ne b}|B_b\cap B_{b'}|=5(r-1)=15,
\]
a contradiction.
\end{proof}

\begin{corollary}\label{cor:forced}
Every Type~I self-dual binary code of length $26$ and minimum distance $6$ has weight enumerator
\[
\begin{aligned}
W_C(x,y)={}&x^{26}+52\,x^{20}y^6+390\,x^{18}y^8+1313\,x^{16}y^{10}+2340\,x^{14}y^{12}\\
& +2340\,x^{12}y^{14}+1313\,x^{10}y^{16}+390\,x^8y^{18}+52\,x^6y^{20}+y^{26}
\end{aligned}
\]
and shadow enumerator
\[
W_S(x,y)=26\,x^{21}y^5+1560\,x^{17}y^9+5020\,x^{13}y^{13}+1560\,x^9y^{17}+26\,x^5y^{21}.
\]
In particular, the shadow contains exactly $26$ vectors of weight $5$.
\end{corollary}

\begin{proof}
Propositions~\ref{prop:parametric-enumerators} and~\ref{prop:no-half} give $t=0$.
\end{proof}

Let
\[
        M=\{x\in S:\wt(x)=5\},
        \qquad
        \mathcal W=\{c\in C:\wt(c)=6\}.
\]
By Corollary~\ref{cor:forced}, $|M|=26$ and $|\mathcal W|=52$.

\begin{lemma}\label{lem:harmonic-design-consequences}
The blocks $B_x$, $x\in M$, form a $1$-design, and so do the blocks $B_c$, $c\in \mathcal W$.  Each coordinate occurs in exactly five blocks of the first family and twelve of the second.
\end{lemma}

\begin{proof}
For $t=0$ one has
\[
        W_{C_0}(x,y)=x^{26}+390\,x^{18}y^8+2340\,x^{14}y^{12}
        +1313\,x^{10}y^{16}+52\,x^6y^{20}.
\]
Moreover, $C_0^\perp=C\sqcup S$ has minimum weight $5$.  Proposition~\ref{prop:four-weight-harmonic}, applied to $D=C_0$, shows that all shells of $C_0^\perp$ are $1$-designs.  The two replication numbers are
\[
        \frac{26\cdot 5}{26}=5,
        \qquad
        \frac{52\cdot 6}{26}=12.
\]
\end{proof}

Put
\[
        M_1=M\cap C_1,
        \qquad
        M_3=M\cap C_3.
\]
We call $M_1$ and $M_3$ the two minimal half-shadows.

\begin{lemma}\label{lem:intersection}
Let $x$ and $y$ be distinct elements of $M$, and set
\[
        k=|B_x\cap B_y|.
\]
Then $k\in\{0,\,1,\,2\}$.  If $x$ and $y$ lie in the same shadow coset, then $k=1$; if they lie in different shadow cosets, then $k\in\{0,\,2\}$.
\end{lemma}

\begin{proof}
One has $\wt(x+y)=10-2\,k$, and $x+y\in C$, so $k\le2$.  In the same shadow coset the sum lies in $C_0$, hence has weight divisible by $4$ and at least $8$, which gives $k=1$.  In different shadow cosets the sum lies in $C_2$, hence has weight $2\pmod4$, which gives $k\in\{0,\,2\}$.
\end{proof}


\section{A projective plane of order 3}\label{sec:projective-plane}

By Lemma~\ref{lem:intersection}, distinct blocks in the same half-shadow meet in one coordinate, whereas blocks from different half-shadows meet in zero or two coordinates.

\begin{proposition}\label{prop:shadow-pairs}
One has $|M_1|=|M_3|=13$.  Every block in either half-shadow meets exactly four blocks in the other half-shadow in two coordinates.  Moreover,
\[
        (x,z)\longmapsto x+z
\]
is a bijection from the $52$ such pairs $(x,z)\in M_1\times M_3$ to $\mathcal W$.
\end{proposition}

\begin{proof}
At least one half-shadow is nonempty; interchange the labels if necessary so that $M_1\ne\varnothing$.  Put $m=|M_1|$, so that $|M_3|=26-m$.  Fix $x\in M_1$, and let $d_1$ be the number of blocks in $M_3$ that meet $B_x$ in two coordinates.  Since the blocks from $M$ form a $1$-design with replication number $5$, Lemma~\ref{lem:intersection} gives
\[
        25=\sum_{u\in M}|B_x\cap B_u|
          =5+(m-1)+2\,d_1.
\]
Thus, $d_1=(21-m)/2$, independently of $x$.  In particular, $m\le21$, so $M_3$ is also nonempty.  Similarly, every block in $M_3$ has
\[
        d_3=\frac{m-5}{2}
\]
partners in $M_1$ that meet it in two coordinates.  Counting these pairs from the two half-shadows gives
\[
        m(21-m)=(26-m)(m-5).
\]
Hence, $m=13$, and then $d_1=d_3=4$.  There are therefore $13\cdot 4=52$ pairs $(x,z)$ with $|B_x\cap B_z|=2$, and every such pair has $x+z\in\mathcal W$.

Conversely, let $c\in\mathcal W$ and put
\[
        j(u)=|B_u\cap B_c|\qquad(u\in M).
\]
Since $c\in C_2$, the word $u+c$ lies in the shadow.  Its weight is $11-2\,j(u)$, and the shadow has no words of weights $1$, $3$, $7$, or $11$.  Hence, $j(u)\in\{1,\,3\}$.  Moreover,
\[
        \sum_{u\in M}j(u)=6\cdot 5=30,
\]
because each coordinate of $B_c$ lies in five blocks from $M$.  It follows that exactly two blocks have $j(u)=3$.  Addition by $c\in C_2$ interchanges these two weight-$5$ words and also interchanges $C_1$ and $C_3$.  They therefore form the unique pair $(x,z)\in M_1\times M_3$ satisfying $c=x+z$.
\end{proof}

For a coordinate $i$, set
\[
        r(i)=|\{x\in M_1:i\in B_x\}|.
\]

\begin{proposition}\label{prop:point-line-stars}
For every coordinate $i$, one has $r(i)\in\{1,\,4\}$, and each value occurs $13$ times.  Put
\[
        P=\{i:r(i)=1\},
        \qquad
        L=\{i:r(i)=4\}.
\]
For each $p\in P$, let $x_p$ denote the unique word in $M_1$ whose support contains $p$.  For each $\ell\in L$, let $z_\ell$ denote the unique word in $M_3$ whose support contains $\ell$.  The maps
\[
        p\longmapsto x_p,
        \qquad
        \ell\longmapsto z_\ell
\]
are bijections from $P$ to $M_1$ and from $L$ to $M_3$, respectively.

Declare $p\in P$ to be incident with $\ell\in L$, and write $p\in\ell$, when
\[
        \ell\in B_{x_p}.
\]
Then
\begin{equation}\label{eq:stars}
\begin{aligned}
        \St(p):=B_{x_p}
        &=\{p\}\cup\{\ell\in L:p\in\ell\},\\
        \St(\ell):=B_{z_\ell}
        &=\{\ell\}\cup\{p\in P:p\in\ell\}.
\end{aligned}
\end{equation}
Moreover,
\begin{equation}\label{eq:star-intersection}
        \St(p)\cap\St(\ell)=
        \begin{cases}
        \{p,\ell\},& p\in\ell,\\
        \varnothing,& p\notin\ell.
        \end{cases}
\end{equation}
\end{proposition}

\begin{proof}
Since the blocks in $M=M_1\cup M_3$ form a $1$-design with replication number $5$, the multiplicity of a coordinate $i$ in $M_3$ is $5-r(i)$.

Fix a coordinate $i$.  There are five blocks of $M$ containing
$i$, and Proposition~\ref{prop:shadow-pairs} gives each of them
four partners in the opposite half-shadow whose supports meet it
in two coordinates.  Counting incidences between these pairs
and their members that contain $i$ therefore gives $20$.

A pair $(x,z)$ in which $i$ belongs to exactly one of $B_x$ and
$B_z$ corresponds, under the bijection
\[
        (x,z)\longmapsto x+z,
\]
to a word in $\mathcal W$ whose support contains $i$.  Since the
words in $\mathcal W$ form a $1$-design with replication number
$12$, there are exactly $12$ such pairs.

There are $r(i)$ blocks of $M_1$ containing $i$ and $5-r(i)$
blocks of $M_3$ containing $i$.  Every pair formed from such two
blocks has nonempty intersection and hence, by
Lemma~\ref{lem:intersection}, has intersection of size $2$.
Thus, there are $r(i)(5-r(i))$ pairs in which $i$ belongs to both
supports, and each contributes twice to the incidence count.
Consequently,
\[
        20=12+2\,r(i)\bigl(5-r(i)\bigr).
\]
Hence,
\[
        r(i)\bigl(5-r(i)\bigr)=4,
\]
and therefore
\[
        r(i)\in\{1,\,4\}.
\]

Since every block of $M_1$ has size $5$,
\[
        65=|M_1|\cdot5=\sum_i r(i)=|P|+4|L|.
\]
Together with $|P|+|L|=26$, this gives
\[
        |P|=|L|=13.
\]

Fix $x\in M_1$.  Each of the other $12$ blocks of $M_1$ meets $B_x$ in one coordinate.  Therefore,
\[
        12
        =\sum_{x'\in M_1\setminus\{x\}}|B_x\cap B_{x'}|
        =\sum_{i\in B_x}\bigl(r(i)-1\bigr).
\]
Each summand is either $0$ or $3$.  Hence, $B_x$ contains exactly one coordinate of $P$ and four coordinates of $L$.

The same argument with $M_1$ and $M_3$ interchanged
(thus replacing $r(i)$ by $5-r(i)$) shows that every $B_z$,
$z\in M_3$, contains four coordinates of $P$ and exactly one
coordinate of $L$.

Now every $p\in P$ lies in a unique block of $M_1$, since $r(p)=1$, and every block of $M_1$ contains a unique coordinate of $P$.  Thus, $x_p$ is well defined and $p\mapsto x_p$ is a bijection from $P$ to $M_1$.  Likewise, every $\ell\in L$ lies in a unique block of $M_3$, since $5-r(\ell)=1$, and $\ell\mapsto z_\ell$ is a bijection from $L$ to $M_3$.

For $p\in P$ and $\ell\in L$, the support $B_{x_p}$ has no coordinate of $P$ other than $p$, while $B_{z_\ell}$ has no coordinate of $L$ other than $\ell$.  Therefore,
\[
        B_{x_p}\cap B_{z_\ell}\subseteq\{p,\ell\}.
\]
By Lemma~\ref{lem:intersection}, this intersection has size $0$ or $2$.  It follows that $\ell$ belongs to $B_{x_p}$ precisely when $p$ belongs to $B_{z_\ell}$.  The descriptions in \eqref{eq:stars} and the intersection formula \eqref{eq:star-intersection} now follow.
\end{proof}

\begin{proposition}\label{prop:PG23}
The incidence structure on $P\sqcup L$ is a projective plane of order $3$.
\end{proposition}

\begin{proof}
Each point lies on four lines and each line contains four points by \eqref{eq:stars}.  If $p\ne q$, then $\St(p)$ and $\St(q)$ meet in exactly one coordinate.  It cannot be a point-coordinate, so it is the unique line through $p$ and $q$.  Dually, two distinct lines meet in a unique point.  Thus, the incidence structure is a symmetric $2$-$(13,4,1)$ design, hence a projective plane of order $3$.
\end{proof}

Every projective plane of order $3$ is isomorphic to $\PG(2,3)$; see \cite[3.2.15]{Dembowski1968}.  We use this uniqueness result in Section~\ref{sec:plane-proof}; the present section and the odd-Golay argument in Section~\ref{sec:golay-proof} require only the incidence and flag properties established here.

An \emph{intrinsic flag} is an incident pair $(p,\ell)$, identified with the two-coordinate set
\[
        F(p,\ell)=\{p,\ell\}=\St(p)\cap\St(\ell).
\]

\begin{proposition}\label{prop:intrinsic-flags}
There are $52$ intrinsic flags, and each coordinate lies in four of them.  The map
\[
        (p,\ell)\longmapsto x_p+z_\ell
\]
is a bijection from the flags to $\mathcal W$.  The intrinsic flags are exactly the two-coordinate sets not contained in the support of any word in $\mathcal W$.
\end{proposition}

\begin{proof}
By Proposition~\ref{prop:point-line-stars}, the intrinsic flags are exactly the pairs of stars whose supports meet in two coordinates.  Proposition~\ref{prop:shadow-pairs} therefore shows that there are $52$ flags and that the displayed map is a bijection onto $\mathcal W$.  Each point is incident with four lines and each line contains four points, so every coordinate lies in four flags.

For the last assertion, write $c=x_q+z_\lambda$, where
$q\in\lambda$, using the displayed bijection.  By
\eqref{eq:stars} and \eqref{eq:star-intersection}, the
point-coordinates of $c$ are the three points of
$\lambda\setminus\{q\}$, while its line-coordinates are the
three lines through $q$ other than $\lambda$.  If $c$ contained
both coordinates of a flag $(p,\ell)$, then
$p\in\lambda$, $p\ne q$, $q\in\ell$, and
$\ell\ne\lambda$.  Since also $p\in\ell$, the distinct points
$p$, $q$ would lie on the two distinct lines $\lambda$, $\ell$, a
contradiction.

Conversely, let $p,q$ be two distinct point-coordinates, let $m$ be their common line, and choose a third point $r$ on $m$.  The word associated with the flag $(r,m)$ contains both $p$ and $q$; the dual argument handles two line-coordinates.  If $p$ and $\ell$ are nonincident, choose $q\in\ell$ and let $m$ be the line joining $p$ and $q$.  The word associated with $(q,m)$ contains both $p$ and $\ell$.  Thus, only intrinsic flags are excluded from all weight-$6$ supports.
\end{proof}

\begin{remark}\rm
Some of the combinatorial counts in this section can alternatively be derived from vector-valued MacWilliams identities and degree-two harmonic weight enumerators.  For the present proof, the direct counting arguments keep the geometry visible and avoid additional notation.
\end{remark}


\section{The plane code}\label{sec:plane-proof}

We recall Glynn's construction from a projective plane of odd order \cite{Glynn1991}.  Let $\Pi$ be a projective plane of order $3$, with point set $P$, line set $L$, incidence matrix $N$, and all-ones matrix $J$.  Over $\F$, put
\begin{equation}\label{eq:plane-code}
        A=J+N,
        \qquad
        C(\Pi)=\{(u,uA):u\in\F^P\}\leq \F^P\oplus\F^L.
\end{equation}
Thus, $A$ is the point-line non-incidence matrix.

\begin{proposition}\label{prop:plane-code}
For every projective plane $\Pi$ of order $3$, the code $C(\Pi)$ is a Type~I self-dual code with parameters $[26,13,6]$.
\end{proposition}

\begin{proof}
Over $\F$ one has
\[
        NN^T=I+J,
        \qquad NJ=JN^T=0,
        \qquad J^2=J,
\]
and hence $AA^T=I$.  Thus, the rows of $[I\mid A]$ are mutually orthogonal.  Their first components are independent, so they span a self-dual code of dimension $13$.

Let $c=(u,uA)\ne0$, set $U=\supp(u)$, $s=|U|$, and
$t=\wt(uA)$.  Since $A$ is the non-incidence matrix, the
$\ell$-coordinate of $uA$ is the parity of the number of points
of $U$ not on $\ell$.  Moreover, if $v=uA$, then
\[
        u=vA^T
\]
because $AA^T=I$.  Thus, after swapping the two coordinate
blocks, the word becomes $(v,vA^T)$ for the dual plane, so the
roles of $s$ and $t$ may be interchanged.

All codewords are even.  If $\wt(c)<6$, then
$s+t\in\{2,4\}$.  Since $A$ is invertible, $s,t>0$; after
interchanging points and lines if necessary, we may therefore
assume $s\le2$.  If $s=1$, then $t=13-4=9$.  If $s=2$, the
contributing lines are precisely those incident with exactly one
of the two points, so
\[
        t=2(4-1)=6.
\]
Both cases contradict $s+t<6$.  Hence, $d(C(\Pi))\ge6$.

Finally, let $U$ consist of three points on a line, and let $q$
be the fourth point of that line.  The contributing lines are
precisely the three lines through $q$ other than the given line.
Thus, $s=t=3$, so $d(C(\Pi))=6$.  Since the code contains a word
of weight $6$, it is Type~I.
\end{proof}

We write
\[
        C^*=C(\PG(2,3)).
\]
A Singer labeling $P=L=\mathbf Z/13\mathbf Z$, with incidence difference set $\{0,\,1,\,3,\,9\}$, makes $A$ circulant with first row
\[
        0010111110111.
\]
This gives the circulant generator matrix recorded by Buyuklieva \cite{Buyuklieva1999}.

\begin{theorem}\label{thm:unique-plane}
Every Type~I self-dual binary code of length $26$ and minimum weight $6$ is equivalent to $C^*$.
\end{theorem}

\begin{proof}
Let $C$ be such a code.  By Proposition~\ref{prop:PG23}, its coordinates are the points and lines of a plane $\Pi\cong\PG(2,3)$.  Let $N$ be its incidence matrix.  The characteristic vector of the point-star $\St(p)$ is
\[
        x_p=(e_p,e_pN).
\]
The vectors $e_p+e_q$ span the even-weight subspace of $\F^P$, and the map $u\mapsto(u,uN)$ is injective.  Hence, the vectors $x_p+x_q\in C_0$ span the $12$-dimensional space
\begin{equation}\label{eq:C0-incidence-form}
        C_0=\{(u,uN):u\in\F^P,\ \wt(u)\equiv0\pmod2\}.
\end{equation}
Here equality follows because $C_0$ also has dimension $12$.

Since every word of $C$ is even, the all-ones vector $\1=(\1_P,\1_L)$ belongs to $C^\perp=C$.  Its weight is $26\equiv2\pmod4$, so it lies outside $C_0$.  Moreover, $\1_P N=0$ because every line contains four points, while $uJ$ is $0$ or $\1_L$ according as $\wt(u)$ is even or odd.  If $u$ has odd weight, write $u=v+\1_P$ with $v$ even.  Equation~\eqref{eq:C0-incidence-form} then gives
\[
        C=C_0\sqcup(C_0+\1)
         =\{(u,u(J+N)):u\in\F^P\}
         =C(\Pi).
\]
Thus, $C$ is equivalent to $C^*$.
\end{proof}

\begin{proposition}\label{prop:aut-plane}
The automorphism group of $C^*$ is
\[
        \Aut(C^*)\cong\PGL(3,3)\rtimes C_2,
        \qquad
        |\Aut(C^*)|=11232.
\]
\end{proposition}

\begin{proof}
By Lemma~\ref{lem:intersection}, two distinct minimal shadow blocks meet in one coordinate exactly when they belong to the same half-shadow.  Thus, the unordered pair $\{M_1,\,M_3\}$ is intrinsic.  Proposition~\ref{prop:point-line-stars} then recovers the distinguished coordinate of every star and the point-line incidence relation.  Every code automorphism consequently induces either a collineation or a duality of $\PG(2,3)$.

Conversely, every collineation preserves the non-incidence matrix $A$ and hence the code $C^*$.  A duality interchanges the point and line coordinates and sends $[I\mid A]$ to $[A^T\mid I]$.  These two matrices have the same row space because $A^{-1}=A^T$.  Hence, all collineations and dualities act on $C^*$.

The collineation group is $\PGL(3,3)$, since  $\mathbb F_3$ has no
nontrivial field automorphism, and a polarity supplies a splitting
element of order $2$. This element acts on $\PGL(3,3)$ by the graph
automorphism induced by point-line duality. Finally,
\[
        |\PGL(3,3)|
        =\frac{(3^3-1)(3^3-3)(3^3-3^2)}{2}
        =5616.
\]
Adjoining a duality therefore gives
\[
        \Aut(C^*)\cong\PGL(3,3)\rtimes C_2,
        \qquad
        |\Aut(C^*)|=11232.
\]
\end{proof}


\section{The odd Golay code descent}\label{sec:golay-proof}

The \emph{odd Golay code}, which we denote by $\mathcal G_{24}^{\mathrm{odd}}$, is the unique Type~I self-dual binary $[24,12,6]$ code \cite{PlessSloane1974,RainsSloaneSelfDual}.
It is a neighbour of the usual extended binary Golay code $\mathcal G_{24}$: for a word $v$ of weight $4$,
\[
 \mathcal G_{24}^{\mathrm{odd}}
 \cong\{g+(g\cdot v)v:g\in\mathcal G_{24}\}.
\]
The two codes share the index-two subcode $\mathcal G_{24}\cap v^\perp$; see \cite[Sec.~11.3]{RainsSloaneSelfDual} (p.~84 of the arXiv version).  The usual $[23,12,7]$ Golay code is obtained by puncturing $\mathcal G_{24}$ in one coordinate.  Here ``odd'' means Type~I, not that codewords have odd weight.

Let $F=F(p,\ell)=\{p,\ell\}$ be an intrinsic flag, and let $\punct_F$ denote puncturing in the two coordinates of $F$.  Put
\[
\begin{aligned}
 H_F&=\{c\in C:c_p=c_\ell\}, & D_F&=\punct_F(H_F),\\
 K_F&=\{c\in C:c_p\ne c_\ell\}, & R_F&=\punct_F(K_F).
\end{aligned}
\]
Thus, $D_F$ is a length-$24$ code and $R_F$ is one of its cosets.

\begin{proposition}[Flag deletion]\label{prop:flag-deletion}
The code $D_F$ is equivalent to $\mathcal G_{24}^{\mathrm{odd}}$, and $R_F$ has minimum weight $5$ and exactly $24$ words of weight $5$.
\end{proposition}

\begin{proof}
The equation $c_p=c_\ell$ is a nonzero linear condition on $C$.  Otherwise $e_p+e_\ell$ would be orthogonal to $C$ and hence belong to $C=C^\perp$, contrary to $d(C)=6$.  Thus, $\dim H_F=12$.  Puncturing on $F$ is injective on $C$, and hence on $H_F$, since a nonzero word in the kernel would have support contained in $F$.  For two words of $H_F$, the two deleted contributions to their inner product are equal and cancel over $\F$.  Therefore, $D_F$ is self-orthogonal of dimension $12$, and hence self-dual.

A nonzero word of $D_F$ of weight below $6$ has weight $2$ or $4$.  Its lift would have weight below $6$ in $C$, except possibly when a weight-$4$ word lifts to a weight-$6$ word containing both coordinates of $F$.  Proposition~\ref{prop:intrinsic-flags} excludes this case.  On the other hand, $x_p+z_\ell$ has weight $6$ and vanishes on $F$.  Thus, $D_F$ is a Type~I self-dual $[24,12,6]$ code, and hence
$D_F\cong\mathcal G_{24}^{\mathrm{odd}}$ by the uniqueness recalled above.

Every word in $K_F$ has exactly one nonzero flag coordinate, so puncturing lowers its weight by one and $R_F$ has minimum weight at least $5$.  By Lemma~\ref{lem:harmonic-design-consequences}, exactly twelve weight-$6$ words contain $p$, and exactly twelve contain $\ell$.  Proposition~\ref{prop:intrinsic-flags} says that these two sets are disjoint.  Their punctures give $24$ distinct words of weight $5$ in $R_F$.  Conversely, every weight-$5$ word of $R_F$ lifts to a weight-$6$ word in one of these two sets.  Hence, $R_F$ has minimum weight $5$ and exactly $24$ words of that weight.
\end{proof}

For $D=\mathcal G_{24}^{\mathrm{odd}}$, we give the required deep hole calculation explicitly, in the orbit-and-coset spirit of \cite{ConwaySloaneCosets}.  Let
$K=\{0,a,b,c\}\cong D_4^\#/D_4$ be the Kleinian four-group, with quadratic form $q(0)=0$ and $q(a)=q(b)=q(c)=1$ in $\F$; see \cite[Sec.~7]{HoehnKleinian} for the $D_4$ interpretation.  Let $\mathcal H=C_6\leq K^6$ be the even self-dual Kleinian hexacode of minimum weight $4$ \cite[Sec.~2]{HoehnKleinian}.  It has $64$ words.  Projection onto any three coordinates is a bijection $\mathcal H\to K^3$, since minimum weight $4$ makes this projection injective.

Label the coordinates of $D$ by $(i,u)$, with $1\leq i\leq6$ and $u\in K$, and write $T_i=\{i\}\times K$ and $t_i=\1_{T_i}$.  A sextet model of $D$ is
\[
 D=\langle t_i+t_j,\ v_h:1\leq i<j\leq6,\ h\in\mathcal H\rangle,
 \qquad \supp(v_h)=\{(i,h_i):1\leq i\leq6\}.
\]
For completeness, these generators are mutually orthogonal because $\mathcal H$ is even.  On the subcode having even weight in every tetrad, the map
$\eta(d)_i=\sum_{u\in K}d_{i,u}u$ has image $\mathcal H$ and kernel the even unions of whole tetrads.  Indeed, the subcode is generated by the $t_i+t_j$ and $v_h+v_0$; the latter map to $h$, and orthogonality to $v_0$ excludes odd unions of tetrads from the kernel.  Its dimension is therefore $6+5=11$, and adjoining $v_0$ gives dimension $12$.  Every nonzero word of this subcode has weight at least $8$: a nonzero image under $\eta$ occupies at least four tetrads, each in two coordinates, and a nonzero word in the kernel is an even union of at least two complete tetrads.  A word of odd weight in each tetrad has weight at least $6$.  Thus, this is a self-dual $[24,12,6]$ code, as required.

The sextet is intrinsic to $D$.  Indeed, its weight-$6$ words are exactly the $v_h$: such a word is a transversal $v_u$, and $\eta(v_u+v_0)=u\in\mathcal H$.  Two distinct coordinates belong to the same tetrad precisely when no weight-$6$ word contains both; for coordinates in different tetrads, the projection property supplies such a word.  Thus, every automorphism of $D$ permutes the tetrads.

Every permutation of $K$ is affine on its underlying two-dimensional $\F$-space.  Hence, an automorphism of $D$ acts on transversal labels by $u\mapsto t+g(u)$, where $t\in K^6$ and $g\in\Aut(K)\wr S_6$.  Since it preserves $\{v_h:h\in\mathcal H\}$, one has $t+g(\mathcal H)=\mathcal H$.  Taking $u=0$ gives $t\in\mathcal H$, and then $g(\mathcal H)=\mathcal H$.  Conversely, translations by $\mathcal H$ and Kleinian automorphisms of $\mathcal H$ preserve the displayed generators of $D$.  Consequently,
\begin{equation}\label{eq:odd-golay-cosets}
 G=\Aut(D)=U\rtimes\Aut_{\mathrm{Kl}}(\mathcal H),
 \qquad U\cong\mathcal H\cong C_2^6,
\end{equation}
where $U$ acts by $(i,u)\mapsto(i,u+h_i)$ for $h\in\mathcal H$, and $\Aut_{\mathrm{Kl}}(\mathcal H)$ is the stabilizer of $\mathcal H$ in $\Aut(K)\wr S_6$.  The Kleinian hexacode group has order $2160$ \cite[Sec.~2]{HoehnKleinian}; it is $3\!\cdot\!S_6$ and induces the full $S_6$ on the six positions \cite{RainsSloaneSelfDual} (p.~15 of the arXiv version).  Thus,
\[
 G\cong2^6\!:\!(3\!\cdot\!S_6),\qquad |G|=64\cdot2160=138240,
\]
recovering the usual sextet group description \cite[Sec.~11.3]{RainsSloaneSelfDual} (p.~84 of the arXiv version).

\begin{lemma}[Deep hole cosets]\label{lem:deep-hole-cosets}
The covering radius of $D$ is $5$.  Its $640$ cosets of minimum weight $5$ form one $G$-orbit.  Their ten $U$-orbits, each of size $64$, correspond to the unordered partitions of the sextet into two triples.  A deep hole coset stabilizer has order $216$.
\end{lemma}

\begin{proof}
First, a word of $D$ supported on at most three tetrads is an even union of whole tetrads.  To see this, let $f_1,f_2,f_3:K\to\F$ describe its restrictions to those tetrads.  Orthogonality to all $v_h$, together with the three-coordinate projection property, gives
$f_1(u)+f_2(v)+f_3(w)=0$ for every $(u,v,w)\in K^3$.  The three functions are constant, with sum zero.

The six tetrad parities of every word of $D$ are equal.  Consequently, tetrad parity induces a surjection
\[
 \bar\pi:\F^{24}/D\longrightarrow
          \F^6/\langle\1\rangle
\]
with $2^{12-5}=128$ cosets in each fibre.  All $t_i$ represent the same nonzero coset $\tau=t_i+D$.
Fix a triple $A$ of tetrads, regarded modulo its complement.  For each of the $64$ transversals $r$ selecting one coordinate in each tetrad of $A$, put $\bar r=r+D$.  The two families
\[
 \bar r,\qquad \bar r+\tau
\]
consist of $64+64=128$ distinct cosets, by the preceding shortening observation; hence these fill the fibre of $A$.  The first family has minimum weight $3$.  In the second family, $r+t_i$ is a representative of weight $5$ for any $i\in A$.  Any weight-$3$ word in this fibre is a transversal of $A$ or of its complement.  A transversal of the complement extends uniquely to some $v_h$, and addition of $v_h$ replaces it by a transversal of $A$.  Thus, all weight-$3$ words lie in the first family; weight $1$ is excluded by parity.  The second family has minimum exactly $5$.

Let $\vareps_i$ be the $i$th standard basis vector of $\F^6$.  For the singleton parity class $[\vareps_i]=\{\vareps_i,\1+\vareps_i\}$, take the four singletons and the four triples in $T_i$, together with the $5\cdot4\cdot6=120$ words consisting of one coordinate in $T_i$ and two in another tetrad.  Their differences are supported on at most three tetrads, so the shortening observation shows that they represent $128$ distinct cosets.  Every coset in this fibre therefore has minimum at most $3$.  Every even class in $\F^6/\langle\1\rangle$ has a parity-vector representative of weight $0$ or $2$.  Adding a unit vector of $\F^{24}$ in a suitable tetrad takes any binary coset over this class to one over a singleton parity class.  The latter has a representative of weight at most $3$, so the original coset has a representative of weight at most $4$.  The deep cosets are therefore precisely the
\[
 \frac12\binom63\,64=640
\]
cosets of minimum $5$ found above, and the covering radius is $5$.

The group $U$ acts regularly on the $64$ transversals of a fixed triple by the projection property, and it fixes $\tau$.  It therefore acts regularly on the $64$ deep cosets in that fibre.  The quotient induces $S_6$, transitive on the ten unordered $3+3$ partitions, so $G$ is transitive on all $640$ deep cosets.  Their stabilizers have order $138240/640=216$.
\end{proof}

Let $R=r+D$ be a deep hole coset, with $\wt(r)=5$, and define
\[
 \alpha_R:D\longrightarrow\F,
 \qquad \alpha_R(d)=r\cdot d.
\]
The functional $\alpha_R$ is independent of the chosen representative $r$ of $R$, since $D=D^\perp$.  Put
\begin{equation}\label{eq:extension-EDR}
 \widehat D_R=\{(\alpha_R(d),d,\alpha_R(d)):d\in D\},
 \qquad
 E(D,R)=\widehat D_R+\langle(1,r,0)\rangle.
\end{equation}
Suppose that $r'=r+e$ is another representative of $R$, with $e\in D$. 
If $\alpha_R(e)=0$, the two choices of extra generator differ by $(0,e,0)\in\widehat D_R$.  If $\alpha_R(e)=1$, then after interchanging the two new coordinates the two choices differ by $(1,e,1)\in\widehat D_R$.
Thus, the equivalence class of $E(D,R)$ depends only on~$R$.

\begin{proposition}[Extension and recovery]\label{prop:golay-extension}
The code $E(D,R)$ is a Type~I self-dual binary $[26,13,6]$ code, and the two new coordinates form an intrinsic flag.  For the pair $(D_F,R_F)$ obtained from a flag of $C$, this construction recovers $C$, after ordering the two restored coordinates suitably.  Moreover, every equivalence of pairs $(D,R)\to(D',R')$ extends to an equivalence $E(D,R)\to E(D',R')$.
\end{proposition}

\begin{proof}
For $d$ and $e$ in $D$, the two new-coordinate contributions to the inner product of their lifts cancel, so $\widehat D_R$ is a self-orthogonal copy of $D$.  Moreover,
\[
 (1,r,0)\cdot(\alpha_R(d),d,\alpha_R(d))
   =\alpha_R(d)+r\cdot d=0,
 \qquad
  (1,r,0)\cdot(1,r,0)=1+\wt(r)=0
\]
in $\F$.  The vector $(1,r,0)$ does not belong to $\widehat D_R$, whose two new coordinates are equal.  Hence, $E(D,R)$ is self-orthogonal of dimension $13$, and therefore self-dual.

A nonzero word in $\widehat D_R$ has weight
\[
        \wt(d)+2\,\alpha_R(d)\ge6,
\]
while a word in the other half has weight
\[
        1+\wt(r+d)\ge6.
\]
The second inequality uses the fact that $R$ has minimum weight $5$.  Equality occurs for $d=0$ in the second half because $\wt(r)=5$.  Thus, $E(D,R)$ has minimum weight $6$ and is Type~I.
A word containing both new coordinates lies in $\widehat D_R$ and has weight at least $8$.  Hence, no weight-$6$ word contains this pair, and Proposition~\ref{prop:intrinsic-flags} identifies it as an intrinsic flag.  Deleting this flag gives back $(D,R)$ by construction.

Now take $(D,R)=(D_F,R_F)$.  Choose a weight-$6$ word $g\in K_F$, order the deleted coordinates so that its deleted part is $(1,0)$, and put $r=\punct_F(g)$.  Each $d\in D_F$ has a unique lift $(\vareps,d,\vareps)\in H_F$, with $\vareps\in\F$.  Self-orthogonality of $C$ gives
\[
        0=g\cdot(\vareps,d,\vareps)=\vareps+r\cdot d.
\]
Thus, $\vareps=\alpha_{R_F}(d)$, so $H_F=\widehat D_{R_F}$ after restoring the two coordinates.  Adjoining $g=(1,r,0)$ recovers the other half of $C$.

Finally, let $\sigma:(D,R)\to(D',R')$ be an equivalence of pairs, and choose $r'=\sigma(r)$.  Acting by $\sigma$ on the middle coordinates and fixing the two new coordinates carries the construction defined by $r$ to the one defined by $r'$.  Choosing another representative of $R'$ changes the result only by the possible interchange of the new coordinates described above.
\end{proof}

\begin{theorem}[Odd-Golay reconstruction]\label{thm:unique-golay}
All codes $E(D,R)$, with $R$ a deep hole coset of the odd Golay code, are equivalent, and every Type~I self-dual $[26,13,6]$ code is recovered from any of its intrinsic flags. Moreover, $\Aut(C)$ is transitive on the $52$ intrinsic flags and
\[
        \Stab_{\Aut(C)}(F)\cong\Stab_{\Aut(D_F)}(R_F).
\]
Consequently,
\[
        |\Aut(C)|=52\cdot 216=11232.
\]
\end{theorem}

\begin{proof}
The odd Golay code is unique and its deep hole cosets form one automorphism orbit, so Proposition~\ref{prop:golay-extension} gives one equivalence class of extensions.  Conversely, the recovery statement in that proposition reconstructs $C$ from every intrinsic flag.

Restriction to the undeleted coordinates gives a homomorphism
\[
        \Stab_{\Aut(C)}(F)\longrightarrow\Stab_{\Aut(D_F)}(R_F).
\]
It is injective: an element in its kernel could only interchange $p$ and $\ell$, but for $g\in K_F$ the sum of $g$ and its image would then be $e_p+e_\ell\in C$, contrary to $d(C)=6$.

Conversely, write $R_F=r+D_F$ and let $\sigma\in\Stab_{\Aut(D_F)}(R_F)$.  Write $\sigma r=r+e$ with $e\in D_F$ and put $\delta=\alpha_{R_F}(e)$.  Since $\sigma^{-1}r+r\in D_F=D_F^\perp$,
\[
        \alpha_{R_F}(\sigma d)=r\cdot\sigma d=\alpha_{R_F}(d)
        \qquad(d\in D_F).
\]
Extend $\sigma$ to the restored coordinates by fixing them if $\delta=0$ and interchanging them if $\delta=1$.  Then $\widehat D_{R_F}$ is preserved and $(1,r,0)$ is sent to
\[
        (1,r,0)+(\delta,e,\delta),
\]
so \eqref{eq:extension-EDR} is preserved.  This proves the stabilizer isomorphism.

For two flags $F$ and $F'$, choose an equivalence $\varphi:D_F\to D_{F'}$ carrying $R_F$ to $R_{F'}$.  Taking $r'=\varphi(r)$, extend $\varphi$ by sending the two restored coordinates for $F$ to the corresponding restored coordinates for $F'$.  This gives an equivalence
$E(D_F,R_F)\longrightarrow E(D_{F'},R_{F'})$.

After applying the recovery identifications, this is an automorphism of $C$ carrying $F$ to $F'$.
Thus, $\Aut(C)$ is flag-transitive.  The stabilizer has order $216$ by Lemma~\ref{lem:deep-hole-cosets}, and there are $52$ flags by Proposition~\ref{prop:intrinsic-flags}.
\end{proof}

\begin{remark}[The deep hole coset stabilizer]\label{rem:deep-stabilizer}
Under the plane identification of Section~\ref{sec:plane-proof}, the common stabilizer in Theorem~\ref{thm:unique-golay} is
\[
        3_+^{1+2}\!:\!D_8,
\]
where $3_+^{1+2}$ is the extraspecial group of order $27$ and exponent $3$, and $D_8$ is dihedral of order $8$.  Indeed, the collineation stabilizer of a point-line flag is the upper-triangular subgroup \hbox{$3_+^{1+2}\!:\!(C_2\times C_2)$,} and a duality interchanging the point and line extends $C_2\times C_2$ to $D_8$.
\end{remark}

\begin{proof}[Proof of Theorem~\ref{thm:main}]
Existence and uniqueness follow from either Section~\ref{sec:plane-proof} or Section~\ref{sec:golay-proof}.  Proposition~\ref{prop:aut-plane} identifies the automorphism group, while Theorem~\ref{thm:unique-golay} independently gives its order and its flag stabilizer from the odd-Golay side.
\end{proof}



\begin{thebibliography}{Hoh26a}

\small

\bibitem[AM69]{AssmusMattson1969}
E.~F.~Assmus, Jr.\ and H.~F.~Mattson, Jr.,
\emph{New 5-designs},
J. Combin. Theory \textbf{6} (1969), 122--151.

\bibitem[Bac99]{Bachoc1999}
C.~Bachoc,
\emph{On harmonic weight enumerators of binary codes},
Des. Codes Cryptogr. \textbf{18} (1999), 11--28.

\bibitem[BG04]{BachocGaborit2004}
C.~Bachoc and P.~Gaborit,
\emph{Designs and self-dual codes with long shadows},
J. Combin. Theory Ser. A \textbf{105} (2004), 15--34.

\bibitem[Bor84]{BorcherdsThesis}
R. E. Borcherds, \emph{The Leech lattice and other lattices},
Ph.D. thesis, Trinity College, Cambridge, 1984;
corrected electronic version, arXiv:math/9911195, 1999.

\bibitem[Buy99]{Buyuklieva1999}
S.~Buyuklieva,
\emph{New binary extremal self-dual codes of lengths 50 and 52},
Serdica Math. J. \textbf{25} (1999), 185--190.

\bibitem[CS90a]{ConwaySloaneCosets}
J.~H.~Conway and N.~J.~A.~Sloane,
\emph{Orbit and coset analysis of the Golay and related codes},
IEEE Trans. Inform. Theory \textbf{36} (1990), 1038--1050.

\bibitem[CS90b]{ConwaySloaneBound}
J.~H.~Conway and N.~J.~A.~Sloane,
\emph{A new upper bound on the minimal distance of self-dual codes},
IEEE Trans. Inform. Theory \textbf{36} (1990), 1319--1333.

\bibitem[CP80]{ConwayPless1980}
J.~H.~Conway and V.~Pless,
\emph{On the enumeration of self-dual codes},
J. Combin. Theory Ser. A \textbf{28} (1980), 26--53.

\bibitem[CPS92]{CPS1992}
J.~H.~Conway, V.~Pless, and N.~J.~A.~Sloane,
\emph{The binary self-dual codes of length up to 32: a revised enumeration},
J. Combin. Theory Ser. A \textbf{60} (1992), 183--195.

\bibitem[Del73]{Delsarte1973}
P.~Delsarte,
\emph{An algebraic approach to the association schemes of coding theory},
Ph.D. thesis, Universit\'e Catholique de Louvain, 1973; Philips Research Reports Supplements \textbf{10} (1973), 1--97.

\bibitem[Dem68]{Dembowski1968}
P.~Dembowski,
\emph{Finite Geometries},
Ergebnisse der Mathematik und ihrer Grenzgebiete, Band~44,
Springer-Verlag, Berlin, 1968.

\bibitem[Dou05]{Dougherty2005}
S.~T.~Dougherty,
\emph{A new construction of self-dual codes from projective planes},
Australas. J. Combin. \textbf{31} (2005), 337--348.

\bibitem[Gly91]{Glynn1991}
D.~G.~Glynn,
\emph{The construction of self-dual binary codes from projective planes of odd order},
Australas. J. Combin. \textbf{4} (1991), 277--284.

\bibitem[Har97]{Harada1997}
M.~Harada,
\emph{Weighing matrices and self-dual codes},
Ars Combin. \textbf{47} (1997), 65--73.

\bibitem[Hoh97]{HoehnShadow}
G.~H\"ohn, \emph{Self-dual vertex operator superalgebras with shadows of large
minimal weight}, Internat. Math. Res. Notices 1997, no.~13, 613--621.

\bibitem[Hoh03]{HoehnKleinian}
G.~H\"ohn,
\emph{Self-dual codes over the Kleinian four group},
Math. Ann. \textbf{327} (2003), 227--255;
\href{https://arxiv.org/abs/math/0005266}{arXiv:math/0005266}.

\bibitem[Hoh08]{HoehnConformal}
G.~H\"ohn,
\emph{Conformal designs based on vertex operator algebras},
Adv. Math. \textbf{217} (2008), 2301--2335.

\bibitem[Hoh26a]{HoehnHexagon26}
G.~H\"ohn,
\emph{The unique extremal threemodular lattice of rank 26, the generalized hexagon $(2,8)$, and the tight Cayley-plane $5$-design},
preprint, 2026; \href{https://arxiv.org/abs/2609.25086}{arXiv:2609.25086}.

\bibitem[Hoh26b]{HoehnAlbert26}
G.~H\"ohn,
\emph{Definite integral Albert algebras: arithmetic and finite geometry},
preprint, 2026; \href{https://arxiv.org/abs/2609.27550}{arXiv:2609.27550}.

\bibitem[Kin01]{King01}
O.~D. King,
\emph{The mass of extremal doubly-even self-dual codes of length $40$},
IEEE Trans. Inform. Theory \textbf{47} (2001), 2558--2560.

\bibitem[MS77]{MS}
F.~J.~MacWilliams and N.~J.~A.~Sloane,
\emph{The Theory of Error-Correcting Codes},
North-Holland, Amsterdam, 1977.

\bibitem[Ple78]{Pless1978}
V.~Pless,
\emph{The children of the $(32,16)$ doubly even codes},
IEEE Trans. Inform. Theory \textbf{24} (1978), 738--746.

\bibitem[PS74]{PlessSloane1974}
V.~Pless and N.~J.~A.~Sloane,
\emph{Binary self-dual codes of length 24},
Bull. Amer. Math. Soc. \textbf{80} (1974), 1173--1178.

\bibitem[PS75]{PlessSloane1975}
V.~Pless and N.~J.~A.~Sloane,
\emph{On the classification and enumeration of self-dual codes},
J. Combin. Theory Ser. A \textbf{18} (1975), 313--335.

\bibitem[Rai98]{Rains1998}
E.~M.~Rains,
\emph{Shadow bounds for self-dual codes},
IEEE Trans. Inform. Theory \textbf{44} (1998), 134--139.

\bibitem[RS98]{RainsSloaneSelfDual}
E.~M.~Rains and N.~J.~A.~Sloane,
\emph{Self-dual codes},
in V.~S. Pless and W.~C. Huffman (eds.), \emph{Handbook of Coding Theory},
Elsevier, Amsterdam, 1998, 177--294;
\href{https://arxiv.org/abs/math/0208001}{arXiv:math/0208001}.


\bibitem[Ven84]{Venkov}
B. B. Venkov, \emph{Even unimodular extremal lattices},
Trudy Mat. Inst. Steklov. \textbf{165} (1984), 43--48;
English transl., Proc. Steklov Inst. Math. \textbf{165} (1985), 47--52.


\end{thebibliography}
\end{document}